\documentclass[12pt]{amsart}
\usepackage[mathscr]{eucal}
\usepackage{amsmath}

\theoremstyle{plain}
	\newtheorem{theorem}{Theorem}[section]
	\newtheorem{lemma}[theorem]{Lemma}
	
	\newtheorem{definition}[theorem]{Definition}
	\newtheorem{proposition}[theorem]{Proposition}
	\newtheorem{remark}[theorem]{Remark}

\theoremstyle{plain}

\def\RN{\mathbb{R}^N}
\def\calF{\mathcal{F}}

\def\e{\varepsilon}
\def\bu{\bar{u}}

\makeatletter
\newcommand{\xRightarrow}[2][]{%
\ext@arrow 0055{\Rightarrowfill@}{#1}{#2}%
}
\def\Rightarrowfill@{\arrowfill@\Relbar\Relbar\Rightarrow}
\newcommand{\xLeftarrow}[2][]{%
\ext@arrow 0055{\Leftarrowfill@}{#1}{#2}%
}
\def\Leftarrowfill@{\arrowfill@\Leftarrow\Relbar\Relbar}
\newcommand{\xLongleftrightarrow}[2][]{%
\ext@arrow 0055{\llrafill@}{#1}{#2}%
}
\def\llrafill@{\arrowfill@\Leftarrow\Relbar\Rightarrow}
\makeatother

\makeatletter
 \@addtoreset{equation}{section}
\makeatother

\begin{document}


\title[Semilinear heat equation]{A doubly critical\\ semilinear heat equation\\
in the $L^1$ space}

\author{Yasuhito Miyamoto}
\address{Graduate School of Mathematical Sciences, The University of Tokyo,
3-8-1 Komaba Meguro-ku Tokyo 153-8914, Japan}
\email{miyamoto@ms.u-tokyo.ac.jp}

\begin{abstract}
We study the existence and nonexistence of a Cauchy problem of the semilinear heat equation
\[
\begin{cases}
\partial_tu=\Delta u+|u|^{p-1}u & \textrm{in}\ \RN\times(0,T),\\
u(x,0)=\phi(x) & \textrm{in}\ \RN
\end{cases}
\]
in $L^1(\RN)$.
Here, $N\ge 1$, $p=1+2/N$ and $\phi\in L^1(\RN)$ is a possibly sign-changing initial function.
Since $N(p-1)/2=1$, the $L^1$ space is scale critical and this problem is known as a doubly critical case.
It is known that a solution does not necessarily exist for every $\phi\in L^1(\RN)$.
Let $X_q:=\{\phi\in L^1_{\rm{loc}}(\RN)\ |\ \int_{\RN}|\phi|\left[\log (e+|\phi|)\right]^qdx<\infty\}(\subset L^1(\RN))$.
In this paper we construct a local-in-time mild solution in $L^1(\RN)$ for $\phi\in X_q$ if $q\ge N/2$.
We show that, for each $0\le q<N/2$, there is a nonnegative initial function $\phi_0\in X_q$ such that the problem has no nonnegative solution, using a necessary condition given by Baras-Pierre [Ann. Inst. H. Poincar\'{e} Anal. Non Lin\'{e}aire {\bf 2} (1985), 185--212].
Since $X_q\subset X_{N/2}$ ($q\ge N/2$), $X_{N/2}$ becomes a sharp integrability condition.
We also prove a uniqueness in a certain set of functions which guarantees the uniqueness of the solution constructed by our method.
\end{abstract}
\date{\today}
\maketitle

\section{Introduction and main results}
We consider the existence and nonexistence of a Cauchy problem of the semilinear heat equation
\begin{equation}\label{S1E1}
\begin{cases}
\partial_tu=\Delta u+|u|^{p-1}u & \textrm{in}\ \RN\times (0,T),\\
u(x,0)=\phi(x) & \textrm{in}\ \RN,
\end{cases}
\end{equation}
where $N\ge 1$, $p=1+2/N$ and $\phi$ is a possibly sign-changing initial function.
When $\phi\in L^{\infty}(\RN)$, one can easily construct a solution by using a fixed point argument.
When $\phi\not\in L^{\infty}(\RN)$, the solvability depends on the balance between the strength of the singularity of $\phi$ and the growth rate of the nonlinearity.
Weissler~\cite{W80} studied the solvability of (\ref{S1E1}), and obtained the following:
\begin{proposition}\label{S1P1}
Let $q_c:=N(p-1)/2$.
Then the following (i) and (ii) hold:\\
(i)(Existence, subcritical and critical cases) Assume either both $q>q_c$ and $q\ge 1$ or $q=q_c>1$. The problem (\ref{S1E1}) has a local-in-time solution for $\phi\in L^q(\RN)$.\\
(ii)(Nonexistence, supercritical case) For each $1\le q<q_c$, there is $\phi\in L^q(\RN)$ such that (\ref{S1E1}) has no local-in-time nonnegative solution.
\end{proposition}
Let $u(x,t)$ be a function such that $u$ satisfies the equation in (\ref{S1E1}).
We consider the scaled function $u_{\lambda}(x,t):=\lambda^{2/(p-1)}u(\lambda x,\lambda^2t)$.
Then, $u_{\lambda}$ also satisfies the same equation.
We can easily see that $\left\|u_{\lambda}(x,0)\right\|_q=\left\|u(x,0)\right\|_q$ if and only if $q=q_c$.
It is well known that $q_c$ is a threshold as Proposition~\ref{S1P1} shows.
However, the case $q=q_c=1$, i.e., $p=1+2/N$, 
is not covered by Proposition~\ref{S1P1}, and it is known that there is a nonnegative initial function $\phi\in L^1(\RN)$ such that (\ref{S1E1}) with $p=1+2/N$ has no local-in-time nonnegative solution.
See Brezis-Cazenave~\cite[Theorem~11]{BC96}, Celik-Zhou~\cite[Theorem~4.1]{CZ03} or Laister~{\it et.al.}$\ $\cite[Corollary~4.5]{LRSV16} for nonexistence results.
See \cite{BP85,HI18,T16} and references therein for existence and nonexistence results with measures as initial data.
In \cite[Section~7.5]{BC96} the case $p=1+2/N$ is referred to as \lq\lq doubly critical case".
Several open problems were given in \cite{BC96}.
It was mentioned in \cite[p.32]{W81} that (\ref{S1E1}) has a local-in-time solution if $\phi\in L^1(\RN)\cap L^q(\RN)$ for some $q>1$.
However, a solvability condition was not well studied.
See Table~1.
For a detailed history about the existence, nonexistence and uniqueness of (\ref{S1E1}), see \cite[Section~1]{CZ03}.
\begin{table}[t]
\begin{tabular}{|c||c|c|c|c|}
\hline
ranges of $q$ & $1\le q<q_c$ & $1=q=q_c$ & $1<q=q_c$ & $q>q_c$, $q\ge 1$\\
 & supercritical & doubly critical & critical & subcritical\\
\hline
\hline
existence/ &not always & not always & exist & exist \\
nonexistence & exist & exist &  & \\
\hline
 & Prop.~\ref{S1P1}~(ii) & $\qquad\ \,$exist: \cite[p.32]{W81}, & Prop.~\ref{S1P1}~(i) & Prop.~\ref{S1P1}~(i)\\
 & & $\qquad\qquad$Thm.~\ref{A}~(i) & & \\
 & & not exist: \cite{BC96,CZ03,LRSV16}, & & \\
 & & $\qquad\qquad\ $Thm.~\ref{A}~(ii) & & \\
\hline
\end{tabular}
\caption{Existence and nonexistence of a local-in-time solution of (\ref{S1E1}) in $L^q(\RN)$.}
\end{table}
\indent
In this paper we obtain a sharp integrability condition on $\phi\in L^1(\RN)$ which determines the existence and nonexistence of a local-in-time solution in the case $p=1+2/N$.
We also show that a solution constructed in Theorem~\ref{A} is unique in a certain set of functions.
Throughout the present paper we define $f(u):=|u|^{p-1}u$.
Let $L^q(\RN)$, $1\le q\le\infty$, denote the usual Lebesgue space on $\RN$ equipped with the norm $\left\|\,\cdot\,\right\|_q$.
For $\phi\in L^1(\RN)$, we define
\[
S(t)[\phi](x):=\int_{\RN}G_t(x-y)\phi(y)dy,
\]
where $G_t(x-y):=(4\pi t)^{-{N}/{2}}\exp\left(-\frac{|x-y|^2}{4t}\right)$.
The function $S(t)[\phi]$ is a solution of the linear heat equation with initial function $\phi$.
We give a definition of a solution of (\ref{S1E1}).
\begin{definition}\label{S1D1}
Let $u$ and $\bu$ be measurable functions on $\RN\times (0,T)$.\\
(i)({Integral solution}) We call $u$ an integral solution of (\ref{S1E1}) if there is $T>0$ such that $u$ satisfies the integral equation
\begin{equation}\label{S1D1E1}
u(t)=\calF[u](t)\ \ \textrm{a.e.}\ x\in\RN,\ \ 0<t<T,\ \ \textrm{and}\ \ \left\|u(t)\right\|_{\infty}<\infty\ \textrm{for}\ 0<t<T,
\end{equation}
where
\[
\calF[u](t):=S(t)\phi+\int_0^tS(t-s)f(u(s))ds.
\]
(ii)({Mild solution}) We call $u$ a mild solution if $u$ is an integral solution and $u(t)\in C([0,T),L^1(\RN))$.\\
(iii) We call $\bu$ a supersolution of (\ref{S1E1}) if $\bu$ satisfies the integral inequality $\calF[\bu](t)\le\bu(t)<\infty$ for a.e. $x\in\RN$, $0<t<T$.
\end{definition}
For $0\le q<\infty$, we define a set of functions by
\[
X_q:=\left\{\phi(x)\in L^1_{\rm{loc}}(\RN)\ \left|\ \int_{\RN}|\phi|\left[\log(e+|\phi|)\right]^qdx<\infty\right.\right\}.
\]
It is clear that $X_q\subset L^1(\RN)$ and that $X_{q_1}\subset X_{q_2}$ if $q_1\ge q_2$.
The main theorem of the paper is the following:
\begin{theorem}\label{A}
Let $N\ge 1$ and $p=1+2/N$.
Then the following (i) and (ii) hold:\\
(i)(Existence) If $\phi\in X_q$ for some $q\ge N/2$, then (\ref{S1E1}) has a local-in-time mild solution $u(t)$, and this mild solution satisfies the following:
\begin{equation}\label{AE0}
\textrm{there is $C>0$ such that $\left\|u(t)\right\|_{\infty}\le Ct^{-\frac{N}{2}}(-\log t)^{-q}$ for small $t>0$}.
\end{equation}
In particular, (\ref{S1E1}) has a local-in-time mild solution for every $\phi\in X_{N/2}$.\\
(ii)(Nonexistence) For each $0\le q<N/2$, there is a nonnegative initial function $\phi_0\in X_q$, which is explicitly given by (\ref{S4E1}), such that (\ref{S1E1}) has no local-in-time nonnegative integral solution, and hence (\ref{S1E1}) has no local-in-time nonnegative mild solution.
\end{theorem}
\begin{remark}
(i) The function $\phi$ in Theorem~\ref{A}~(i) is not necessarily nonnegative.\\
(ii) Theorem~\ref{A} indicates that $X_{N/2}(\subset L^1(\RN))$ is an optimal set of initial functions for the case $p=1+2/N$, and $X_{N/2}$ is slightly smaller than $L^1(\RN)$.
This situation is different from the case $p>1+2/N$, since (\ref{S1E1}) is always solvable in the scale critical space $L^{N(p-1)/2}$ for $p>1+2/N$ (Proposition~\ref{S1E1}~(i)).\\
(iii) $L^1(\RN)$ is larger than the optimal set for $p=1+2/N$.
On the other hand, it follows from Proposition~\ref{S1P1}~(i) that if $1<p<1+2/N$, then (\ref{S1E1}) has a solution for all $\phi\in L^1(\RN)$.\\
Therefore, $L^1(\RN)$ is small enough for the case $1<p<1+2/N$.\\
(iv) The function $\phi_0$ given in Theorem~\ref{A}~(ii) is modified from $\psi(x)$ given by (\ref{S1E5}).
This function comes from Baras-Pierre~\cite{BP85}, and Theorem~\ref{A}~(ii) is a rather easy consequence of \cite[Proposition~3.2]{BP85}.
However, we include Theorem~\ref{A}~(ii) for a complete description of the borderline property of $X_{N/2}$.\\
(v) Laister {\it et.al.}~\cite{LRSV16} obtained a necessary and sufficient condition for the existence of a local-in-time nonnegative solution of
\begin{equation}\label{S1E2-}
\begin{cases}
\partial_tu=\Delta u+h(u) & \textrm{in}\ \RN\times(0,T),\\
u(x,0)=\phi(x)\ge 0 & \textrm{in}\ \RN.
\end{cases}
\end{equation}
They showed that when $h(u)=u^{1+2/N}[\log(e+u)]^{-r}$, (\ref{S1E2-}) has a local-in-time nonnegative solution for every nonnegative $\phi\in L^1(\RN)$ if $1<r<\lambda p$, and (\ref{S1E2-}) does not always have if $0\le r\le1$.
Here, $\lambda>0$ is a certain constant.
Therefore, the optimal growth of $h(u)$ for $L^1(\RN)$ is slightly smaller than $u^{1+2/N}$.\\
(vi) The exponent $p=1+2/N$, which is called Fujita exponent, also plays a key role in the study of global-in-time solutions.
If $1<p\le 1+2/N$, then every nontrivial nonnegative solution of (\ref{S1E1}) blows up in a finite time.
If $p>1+2/N$, then (\ref{S1E1}) has a global-in-time nonnegative solution.
See Fujita~\cite{F66}.
In particular, in the case $p=1+2/N$ we cannot expect a global existence of a classical solution for small initial data.
\end{remark}
The next theorem is about the uniqueness of the integral solution in a certain class.
\begin{theorem}\label{B}
Let $N\ge 1$, $p=1+2/N$ and $q>N/2$.
Then an integral solution $u(t)$ of (\ref{S1E1}) is unique in the set
\begin{equation}\label{BE0}
\left\{u(t)\in L^1(\RN)\ \left|\ \sup_{0\le t\le T}t^{N/2}(-\log t)^q\left\|u(t)\right\|_{\infty}<\infty\right.\right\}.
\end{equation}
Therefore, a solution given by Theorem~\ref{A} is unique.
\end{theorem}
\begin{remark}
(i) If there were a solution that does not satisfy (\ref{BE0}), then the uniqueness fails.
However, it seems to be an open problem.\\
(ii) In the case $q=N/2$ the uniqueness under (\ref{BE0}) is left open.\\
(iii) For general $p$ and $q$, the uniqueness of a solution of (\ref{S1E1}) is known in the set
\[
\left\{u(t)\in L^q(\RN)\ \left|\ \sup_{0\le t\le T}t^{\frac{N}{2}\left(\frac{1}{q}-\frac{1}{pq}\right)}\left\|u(t)\right\|_{pq}<\infty\right.\right\}.
\]
See Haraux-Weissler~\cite{HW82} and \cite{W80}.
For an unconditional uniqueness with a certain range of $p$ and $q$, see \cite[Theorem~4]{BC96}.\\
(iv) The nonuniqueness in $L^q(\RN)$ is also known for (\ref{S1E1}).
For $p>1+2/N$ and $1\le q<N(p-1)/2<p+1$, see \cite{HW82}.
For $p=q=N/(N-2)$, see Ni-Sacks~\cite{NS85} and Terraneo~\cite{T02}.
\end{remark}

Let us mention technical details.
We assume that $\phi\in X_q$ for some $q\ge N/2$.
Using a monotone method, we construct a nonnegative mild solution $w(t)$ of
\begin{equation}\label{S1E2}
\begin{cases}
\partial_tw=\Delta w+f(w) & \textrm{in}\ \RN\times (0,T),\\
w(x,0)=|\phi(x)| & \textrm{in}\ \RN.
\end{cases}
\end{equation}
We define $g(u)$ by
\begin{equation}\label{g}
g(u):=u\left[\log(\rho+|u|)\right]^q,
\end{equation}
where $\rho>1$ is chosen appropriately.
We will see that if $\rho\ge e$, then $g(u)$ is convex for $u\ge 0$ and $g$ plays a crucial role in the construction of the solution of (\ref{S1E2}).
In order to construct a nonnegative solution we use a method developed by Robinson-Sier\.{z}\c{e}ga~\cite{RS13} with the convex function $g$, which was also used in Hisa-Ishige~\cite{HI18}.
We define a sequence of functions $(u_n)_{n=0}^{\infty}$ by
\begin{equation}\label{S1E3}
\begin{cases}
u_n(t)=\calF[u_{n-1}](t)\ \textrm{for}\ 0\le t<T & \textrm{if}\ n\ge 1,\\
u_0(t)=0. & {}
\end{cases}
\end{equation}
Then, we show that $-w(t)\le u_n(t)\le w(t)$ for $0\le t<T$.
Since $|u_n(t)|\le w(t)$, we can extract a convergent subsequence in $C_{\rm{loc}}(\RN\times (0,T))$, using a parabolic regularization, the dominated convergence theorem and a diagonal argument.
The limit function becomes a mild solution (\ref{S1E1}).

In the nonexistence part we use a necessary condition for the existence of a nonnegative solution of (\ref{S1E1}) obtained by Baras-Pierre~\cite{BP85}, which is stated in Proposition~\ref{S2P2} in the present paper.
Using their result, one can show that there is $c_0>0$ such that if $\phi(x)\ge c_0\psi(x)$
in a neighborhood of the origin, then (\ref{S1E1}) has no nonnegative integral solution.
Here,
\begin{equation}\label{S1E5}
\psi(x):=|x|^{-N}\left(-\log|x|\right)^{-\frac{N}{2}-1}\ \ \textrm{for}\ \ 0<|x|<1/e.
\end{equation}
See also \cite{HI18}.
For each $0\le q<N/2$ we will see that a modified function $\phi_0$, which is given by (\ref{S4E1}), belongs to $X_q$.
We show that $\phi_0$ does not satisfy the necessary condition for the existence of an integral solution stated in  Proposition~\ref{S2P2}.
Hence, (\ref{S1E1}) with $\phi_0$ has no nonnegative solution for each $0\le q<N/2$.

This paper consists of five sections.
In Section~2 we recall known results including a monotone method, a necessary condition on the existence for (\ref{S1E1}) and $L^p$-$L^q$-estimates.
In Section~3 we prove Theorem~\ref{A}~(i).
In Section~4 we prove Theorem~\ref{A}~(ii).
In Section~5 we prove Theorem~\ref{B}.

\section{Preliminaries}
First we recall the monotonicity method.
\begin{lemma}\label{S2L1}
Let $0<T\le\infty$ and let $f$ be a continuous nondecreasing function such that $f(0)\ge 0$.
The problem (\ref{S1E1}) has a nonnegative integral solution for $0<t<T$ if and only if (\ref{S1E1}) has a nonnegative supersolution for $0<t<T$.
Moreover, if a nonnegative supersolution $\bu(t)$ exists, then the solution $u(t)$ obtained in this lemma satisfies $0\le u(t)\le\bu(t)$.
\end{lemma}
\begin{proof}
This lemma is well known.
See \cite[Theorem~2.1]{RS13} for details.
However, we briefly show the proof for readers' convenience.

If (\ref{S1E1}) has an integral solution, then the solution is also a supersolution.
Thus, it is enough to show that (\ref{S1E1}) has an integral solution if (\ref{S1E1}) has a supersolution.
Let $\bu$ be a supersolution for $0<t<T$.
Let $u_1=S(t)\phi$.
We define $u_n$, $n=2,3,\ldots$, by
\[
u_n=\calF[u_{n-1}].
\]
Then we can show by induction that
\[
0\le u_1\le u_2\le\cdots\le u_n\le\cdots\le\bu <\infty\ \ \textrm{a.e.}\ x\in\RN,\ 0<t<T.
\]
This indicates that the limit $\lim_{n\to\infty}u_n(x,t)$ which is denoted by $u(x,t)$ exists for almost all $x\in\RN$ and $0<t<T$.
By the monotone convergence theorem we see that
\[
\lim_{n\to\infty}\calF[u_{n-1}]=\calF[u],
\]
and hence $u=\calF[u]$.
Then, $u$ is an integral solution of (\ref{S1E1}).
It is clear that $0\le u(t)\le \bu(t)$.
\end{proof}
Baras-Pierre~\cite{BP85} studied necessary conditions for the existence of an integral solution in the case $p>1$.
See also \cite{HI18} for details of necessary conditions including Proposition~\ref{S2P2}.
The following proposition is a variant of \cite[Proposition~3.2]{BP85}.
\begin{proposition}\label{S2P2}
Let $N\ge 1$ and $p=1+2/N$. 
If $u(t)$ is an nonnegative integral solution, i.e., $u(t)$ satisfies (\ref{S1D1E1}) with a nonnegative initial function $\phi$ and some $T>0$, then there exists a constant $\gamma_0>0$ depending only on $N$ and $p$ such that
\begin{equation}\label{S2P2E0}
\int_{B(\tau)}\phi(x)dx\le\gamma_0|\log\tau|^{-\frac{N}{2}}
\ \ \textrm{for all}\ \ 0<\tau<T,
\end{equation}
where $B(\tau):=\{x\in\RN\ |\ |x|<\tau\}$.
\end{proposition}

\begin{lemma}\label{S2L2}
Let $q\ge 0$ be fixed, and let
\begin{equation}\label{S2L2E0}
X_{q,\rho}:=\left\{\phi\in L^1(\RN)\ \left|\ \int_{\RN}|\phi|\left[\log(\rho+|\phi|)\right]^qdx<\infty\right.\right\}.
\end{equation}
Then, $\phi\in X_{q,\rho}$ for all $\rho>1$ if and only if $\phi\in X_{q,\sigma}$ for some $\sigma>1$.
\end{lemma}
\begin{proof}
We consider only the case $q>0$.
It is enough to show that $\phi\in X_{q,\rho}$ for all $\rho>1$ if $\phi\in X_{q,\sigma}$ for some $\sigma>1$.
Let $\rho>1$ be fixed, and let $\xi(s):=\log(\rho+s)/(\log(\sigma+s))$.
By L'Hospital's rule we see that $\lim_{s\to\infty}\xi(s)=\lim_{s\to\infty}(s+\sigma)/(s+\rho)=1$.
Since $\xi(s)$ is bounded on each compact interval in $[0,\infty)$, we see that $\xi(s)$ is bounded in $[0,\infty)$, and hence there is $C>0$ such that $\log (\rho+s)\le C\log (\sigma+s)$ for $s\ge 0$.
This inequality indicates that $\phi\in X_{q,\rho}$ if $\phi\in X_{q,\sigma}$.
\end{proof}
Because of Lemma~\ref{S2L1}, we do not care about $\rho>1$ in (\ref{S2L2E0}).
In particular, if $\phi\in X_q$, then $\left\|g(\phi)\right\|_1<\infty$ for every $\rho>1$.

\begin{proposition}\label{S2P1}
(i) Let $N\ge 1$ and $1\le\alpha\le\beta\le\infty$.
There is $C>0$ such that, for $\phi\in L^{\alpha}(\RN)$,
\[
\left\|S(t)\phi\right\|_{\beta}
\le {C}{t^{-\frac{N}{2}\left(\frac{1}{\alpha}-\frac{1}{\beta}\right)}}
\left\|\phi\right\|_{\alpha}
\ \ \textrm{for}\ \ t>0.
\]
(ii) Let $N\ge 1$ and $1\le\alpha<\beta\le\infty$.
Then, for each $\phi\in L^{\alpha}(\RN)$ and $C_0>0$, there is $t_0=t_0(C_0,\phi)$ such that
\[
\left\|S(t)\phi\right\|_{\beta}\le C_0t^{-\frac{N}{2}\left(\frac{1}{\alpha}-\frac{1}{\beta}\right)}
\ \ \textrm{for}\ \ 0<t<t_0.
\]
\end{proposition}
For Proposition~\ref{S2P1}~(i) (resp. (ii)), see \cite[Proposition~48.4]{QS07} (resp. \cite[Lemma~8]{BC96}).
Note that $C_0>0$ in (ii) can be chosen arbitrary small.

We collect various properties of $g$ defined by (\ref{g}).
\begin{lemma}\label{S2L3}
Let $q>0$ and let $g_1(s):=s[\log(\rho+s)]^{-q}$.
Then the following hold:\\
(i) If $\rho>1$, then $g'(s)>0$ for $s>0$.\\
(ii) If $\rho\ge e$, then $g''(s)>0$ for $s>0$.\\
(iii) If $\rho\ge e$, then $g_1(s)\le g^{-1}(s)$ for $s\ge 0$.\\
(iv) If $\rho>1$, then there is $C_1>0$ such that $g^{-1}(s)\le g_1(C_1s)$ for $s\ge 0$.\\
(v) If $\rho>e^{q/(p-1)}$, then $g^{-1}(s)^p/s$ is nondecreasing for $s\ge 0$.\\
(vi) If $\rho\ge e$, then, for $\phi\in L^1(\RN)$,
\[
S(t)\phi\le g^{-1}(S(t)g(\phi))\ \ \textrm{for}\ \ t\ge 0.
\]
\end{lemma}
\begin{proof}
By direct calculation we have
\begin{align*}
g'(s)&=[\log(\rho+s)]^{q-1}\left\{\log (\rho+s)+\frac{q s}{s+\rho}\right\},\\
g''(s)&=\frac{q[\log(s+\rho)]^{q-2}}{(s+\rho)^2}\left[s\left\{\log(\rho+s)+q-1\right\}+2\rho\log(\rho+s)\right].
\end{align*}
Thus, (i) and (ii) hold.\\
(iii) Since $\rho\ge e$, we have
\begin{equation}\label{S3L3E1}
g(g_1(s))=\frac{s}{[\log(\rho+s)]^q}\left[\log\left(\rho+\frac{s}{[\log(\rho+s)]^q}\right)\right]^q
\le\frac{s}{[\log(\rho+s)]^q}[\log(\rho+s)]^q=s
\end{equation}
for $s\ge 0$.
By (i) we see that $g^{-1}(s)$ exists and it is increasing.
By (\ref{S3L3E1}) we see that $g_1(s)\le g^{-1}(s)$ for $s\ge 0$.\\
(iv) Let $\xi(s):=(g(g_1(s))/s)^{1/q}=\log(\rho+\frac{s}{[\log(\rho+s)]^q})/(\log(\rho+s))$.
Then, for each compact interval $I\subset [0,\infty)$, there is $c>0$ such that $\xi(s)>c$ for $s\in I$.
By L'Hospital's rule we have
\[
\lim_{s\to\infty}\xi(s)=\lim_{s\to\infty}
\frac{1+\frac{\rho}{s}}{1+\frac{\rho}{s}[\log(\rho+s)]^q}\left\{1-\frac{1}{1+\frac{\rho}{s}}\frac{q}{\log(\rho+s)}\right\}=1,
\]
and hence there is $c_0>0$ such that $\xi(s)\ge c_0$ for $s\ge 0$.
Thus, $g^{-1}(c_0^qs)\le g_1(s)$ for $s\ge 0$.
Then, the conclusion holds.\\
(v) By (i) we see that $g(\tau)$ is increasing.
Let $s:=g(\tau)$.
Then, ${g^{-1}(s)^p}/{s}=\tau^{p-1}\left[\log(\rho+\tau)\right]^{-q}$.
Since $\rho>e^{q/(p-1)}$, we have
\[
\frac{d}{d\tau}\frac{\tau^{p-1}}{[\log(\rho+\tau)]^q}=
\frac{\tau^{p-2}}{[\log(\rho+\tau)]^{q+1}}\left\{(p-1)\log(\rho+\tau)-\frac{q\tau}{\rho+\tau}\right\}>0.
\]
Thus, $g^{-1}(s)^p/s$ is increasing for $s\ge 0$.\\
(vi) Because of (ii), $g$ is convex.
By Jensen's inequality we see that $g(S(t)\phi)\le S(t)g(\phi)$.
Since $g^{-1}$ exists and $g^{-1}$ is increasing, the conclusion holds.
The proof is complete.
\end{proof}

\section{Existence}
\begin{lemma}\label{S3L1}
Let $N\ge 1$ and $p=1+2/N$.
Assume that $\phi\ge 0$.
If $\phi\in X_q$ for some $q\ge N/2$, then (\ref{S1E1}) has a local-in-time nonnegative mild solution $u(t)$, and $\left\|u(t)\right\|_{\infty}\le Ct^{-N/2}(-\log t)^{-q}$ for small $t>0$.
\end{lemma}
\begin{proof}
First, we consider the case $q=N/2$.
Let $\rho\ge\max\{e^{q/(p-1)},e\}$ be fixed.
Let $g$ be defined by (\ref{g}).
Here, $q=N/2$ and $g$ satisfies Lemma~\ref{S2L3}.
We define
\[
\bu(t):=2g^{-1}(S(t)g(\phi)).
\]
We show that $\bu$ is a supersolution.
By Lemma~\ref{S2L3}~(vi) we have
\begin{equation}\label{S3L1E1}
S(t)\phi\le g^{-1}\left(S(t)g(\phi)\right)=\frac{\bu(t)}{2}.
\end{equation}
Next, we have
\begin{align}
\int_0^tS(t-s)f(\bu(s))ds
&=2^p\int_0^tS(t-s)\left[S(s)g(\phi)\frac{g^{-1}\left(S(s)g(\phi)\right)^p}{S(s)g(\phi)}\right]ds\nonumber\\
&\le 2^pS(t)g(\phi)\int_0^t\left\|\frac{g^{-1}\left(S(s)g(\phi)\right)^p}{S(s)g(\phi)}\right\|_{\infty}ds\nonumber\\
&\le 2^pg^{-1}\left(S(t)g(\phi)\right)
\left\|\frac{S(t)g(\phi)}{g^{-1}\left(S(t)g(\phi)\right)}\right\|_{\infty}
\int_0^t\left\|\frac{g^{-1}\left(S(s)g(\phi)\right)^p}{S(s)g(\phi)}\right\|_{\infty}ds.\label{S3L1E2}
\end{align}
Since $g(\phi)\in L^1(\RN)$, by Proposition~\ref{S2P1}~(ii) we have
\begin{equation}\label{S3L1E3}
\left\|S(t)g(\phi)\right\|_{\infty}\le C_0t^{-N/2}.
\end{equation}
By Lemma~\ref{S2L3}~(v) we see that $g^{-1}(u)^p/u$ is nondecreasing for $u\ge 0$.
Using (\ref{S3L1E3}) and Lemma~\ref{S2L3}~(iv), we have
\begin{multline}\label{S3L1E4}
\left\|\frac{g^{-1}\left(S(s)g(\phi)\right)^p}{S(s)g(\phi)}\right\|_{\infty}
\le\frac{g^{-1}\left(\left\|S(s)g(\phi)\right\|_{\infty}\right)^p}{\left\|S(s)g(\phi)\right\|_{\infty}}\\
\le \frac{g^{-1}(C_0s^{-N/2})^p}{C_0s^{-N/2}}
\le\frac{C_1^pC_0^{2/N}}{s\left[\log\left(\rho+C_0C_1s^{-N/2}\right)\right]^{pq}}
\le\frac{C_0^{2/N}C_1'}{s(-\log s)^{pq}}
\end{multline}
for $0<s<s_0(C_0)$, where $C_1'$ is a constant independent of $C_0$.
Using Lemma~\ref{S2L3}~(iii) and (\ref{S3L1E3}), we have
\begin{multline}\label{S3L1E5}
\left\|\frac{S(t)g(\phi)}{g^{-1}\left(S(t)g(\phi)\right)}\right\|_{\infty}
\le\left\|\frac{S(t)g(\phi)}{g_1(S(t)g(\phi))}\right\|_{\infty}
=\left\|\left[\log(\rho+S(t)g(\phi))\right]^q\right\|_{\infty}\\
\le\left[\log(\rho+\left\|S(t)g(\phi)\right\|_{\infty})\right]^q
\le\left[\log(\rho+C_0t^{-N/2})\right]^q
\le C_2'(-\log t)^q
\end{multline}
for $0<t<t_0(C_0)$, where $g_1$ is defined in Lemma~\ref{S2L3} and $C_2'$ is a constant independent of $C_0$.
By (\ref{S3L1E4}) and (\ref{S3L1E5}) we have
\begin{multline}\label{S3L1E6}
\left\|\frac{S(t)g(\phi)}{g^{-1}\left(S(t)g(\phi)\right)}\right\|_{\infty}
\int_0^t\left\|\frac{g^{-1}\left(S(s)g(\phi)\right)^p}{S(s)g(\phi)}\right\|_{\infty}ds
\le C_0^{2/N}C_1'C_2'(-\log t)^q\int_0^t\frac{ds}{s(-\log s)^{pq}}\\
=C_0^{2/N}C_1'C_2'(-\log t)^q\frac{2}{N(-\log t)^q}
=C_0^{2/N}C_1'C_2'\frac{2}{N}
\end{multline}
for $0<t<\min\{s_0(C_0),t_0(C_0)\}$.
By Proposition~\ref{S2P1}~(ii) we can take $C_0>0$ such that $2^{p+1}C_0^{2/N}C_1'C_2'/N<1$.
By (\ref{S3L1E1}), (\ref{S3L1E2}) and (\ref{S3L1E6}) we have
\[
\calF[\bu](t)
=S(t)\phi+\int_0^tS(t-s)f(\bu(s))ds
\le\frac{1}{2}\bu(t)+\frac{1}{2}\bu(t)=\bu(t)
\]
for small $t>0$.
Thus, there is $T>0$ such that $\calF[\bu]\le\bu$ for $0<t<T$, and hence $\bu$ is a supersolution.
By Lemma~\ref{S2L1} we see that there is $T>0$ such that (\ref{S1E1}) has a solution for $0<t<T$, and $u(t)$ is clearly nonnegative.
Moreover,
\begin{equation}\label{S3L1E6+}
0\le u(t)\le\bu(t)=2g^{-1}(S(t)g(\phi))\le Ct^{-\frac{N}{2}}(-\log t)^{-q},
\end{equation}
which is the estimate in the assertion.
We show that $u(t)\in C([0,T),L^1(\RN))$.
Since $\left\|g^{-1}(u)\right\|_1\le C\left\|u\right\|_1$, by (\ref{S3L1E6}) and Proposition~\ref{S2P1}~(i) we have
\begin{multline}\label{S3L1E7-}
\left\| u(t)-S(t)\phi\right\|_1
\le\left\|\int_0^tS(t-s)f(\bu (s))ds\right\|_1
\le C_0^{2/N}C_1'C_2'\frac{2}{N}\left\| g^{-1}(S(t)g(\phi))\right\|_1\\
\le C_0^{2/N}C_1'C_2'\frac{2}{N}C\left\| S(t)g(\phi)\right\|_1
\le C_0^{2/N}C_1'C_2'\frac{2}{N}C'\left\| g(\phi)\right\|_1
\end{multline}
for small $t>0$, where $C'$ is independent of $C_0$.
By Proposition~\ref{S2P1}~(ii) we can take $C_0>0$ arbitrary small, and hence
\[
\left\|u(t)-S(t)\phi\right\|_1
\to 0\ \ \textrm{as}\ \ t\downarrow 0.
\]
Since $S(t)$ is a strongly continuous semigroup on $L^1(\RN)$ (see e.g., \cite[Section 48.2]{QS07}), we have
\begin{equation}\label{S3L1E7}
\left\|u(t)-\phi\right\|_1\le\left\|u(t)-S(t)\phi\right\|_1+\left\|S(t)\phi-\phi\right\|_1
\to 0\ \ \textrm{as}\ \ t\downarrow 0.
\end{equation}
It follows from (\ref{S3L1E2}) and (\ref{S3L1E6}) that $\left\|\int_0^tS(t-s)f(\bu(s))ds\right\|_1<\infty$ for $0<t<T$.
We see that if $0<t<T$, then
\begin{equation}\label{S3L1E8}
\left\|u(t+h)-u(t)\right\|_1\to 0\ \ \textrm{as}\ \ h\to 0.
\end{equation}
By (\ref{S3L1E7}) and (\ref{S3L1E8}) we see that $u(t)\in C([0,T),L^1(\RN))$.
The proof of (i) is complete.

Next, we consider the case $q>N/2$.
The argument is the same until (\ref{S3L1E6}).
We have
\begin{multline}\label{S3L1E9}
\left\|\frac{S(t)g(\phi)}{g^{-1}\left(S(t)g(\phi)\right)}\right\|_{\infty}
\int_0^t\left\|\frac{g^{-1}\left(S(s)g(\phi)\right)^p}{S(s)g(\phi)}\right\|_{\infty}ds
\le C_0^{2/N}C_1'C_2'(-\log t)^q\int_0^t\frac{ds}{s(-\log s)^{pq}}\\
=\frac{C_1^{2/N}C_1'C_2'}{pq-1}(-\log t)^{1-\frac{2q}{N}}
\end{multline}
instead of (\ref{S3L1E6}).
Since the RHS of (\ref{S3L1E9}) goes to $0$ as $t\downarrow 0$, the rest of the proof is almost the same with obvious modifications.
In particular, (\ref{S3L1E6+}) holds even for $q>N/2$.
We omit the details.
\end{proof}

We consider (\ref{S1E2}), where $\phi$ is given in (\ref{S1E1}).
By Lemma~\ref{S3L1} we see that (\ref{S1E2}) has a local-in-time solution which is denoted by $w(t)$.
We consider the sequence $(u_n)_{n=0}^{\infty}$ defined by (\ref{S1E3}).
Then, the following lemma says that $\left\|u_n(t)\right\|_{\infty}$ can be controlled by $w(t)$.
\begin{lemma}\label{S3L2}
Let $u_n$ be as defined by (\ref{S1E3}), and let $w$ be a solution of (\ref{S1E2}) on $(0,T)$.
Then,
\begin{equation}\label{S3L2E0}
-w(t)\le u_n(t)\le w(t)\ \ \textrm{for a.e.}\ x\in\RN\ \textrm{and}\ 0<t<T.
\end{equation}
\end{lemma}
\begin{proof}
It is clear from the definitions of $u_0$ and $w(t)$ that
\[
u_0(t)\le w(t)\ \ \textrm{for}\ \ 0<t<T.
\]
We assume that $u_{n-1}(t)\le w(t)$ on $(0,T)$.
Then, we have
\[
w(t)
=S(t)|\phi|+\int_0^tS(t-s)f(w(s))ds\\
\ge S(t)\phi+\int_0^tS(t-s)f(u_{n-1}(s))ds\\
=u_n(t),
\]
and hence $u_n(t)\le w(t)$ for $0<t<T$.
Thus, by induction we see that, for $n\ge 0$,
\begin{equation}\label{S3L2E1}
u_n(t)\le w(t)\ \ \textrm{on}\ \ 0<t<T.
\end{equation}

It is clear that $u_0(t)\ge -w(t)$ for $0<t<T$.
We assume that $u_{n-1}(t)\ge -w(t)$ on $(0,T)$.
Then, we have
\[
u_n(t)
=S(t)\phi+\int_0^tS(t-s)f(u_{n-1}(s))ds
\ge -S(t)|\phi|+\int_0^tS(t-s)f(-w(s))ds
=-w(t),
\]
and hence, $u_n(t)\ge -w(t)$ on $(0,T)$.
Thus, by induction we see that for $n\ge 0$,
\begin{equation}\label{S3L2E2}
-w(t)\le u_n(t)\ \ \textrm{on}\ \ 0<t<T.
\end{equation}
By (\ref{S3L2E1}) and (\ref{S3L2E2}) we see that (\ref{S3L2E0}) holds.
\end{proof}

\begin{proof}[Proof of Theorem~\ref{A}]
(i) Let $(u_n)_{n=0}^{\infty}$ be defined by (\ref{S1E3}).
Using an induction argument with a parabolic regularity theorem, we can show that, for each $n\ge 1$, $u_n\in C^{2,1}(\RN\times (0,T))$ and $u_n$ satisfies the equation 
\[
\partial_tu_n=\Delta u_n+f(u_{n-1})\ \ \textrm{in}\ \ \RN\times (0,T)
\]
in the classical sense.
Let $K$ be an arbitrary compact subset in $\RN\times(0,T)$, and let $K_1$, $K_2$ be two compact sets such that $K\subset K_1\subset K_2\subset\RN\times (0,T)$.
Because of Lemma~\ref{S3L2}, $f(u_{n-1})$ is bounded in $C(K_2)$.
By a parabolic regularity theorem we see that $u_n$ is bounded in $C^{\gamma,\gamma/2}(K_1)$.
Using a parabolic regularity theorem again, we see that $u_{n+1}$ is bounded in $C^{2+\gamma,1+\gamma/2}(K)$.

In the following we use a diagonal argument to obtain a convergent subsequence in $\RN\times (0,T)$.
Let $Q_j:=\overline{\{x\in\RN|\ |x|\le j\}}\times \left[\frac{T}{j+2},\frac{(j+1)T}{j+2}\right]$.
Since $(u_n)_{n=3}^{\infty}$ is bounded in $C^{2,1}(Q_1)$, by Ascoli-Arzer\`{a} theorem there is a subsequence $(u_{1,k})\subset (u_n)$ and $u_1^*\in C(Q_1)$ such that $u_{1,k}\to u_1^*$ in $C(Q_1)$ as $k\to\infty$.
Since $(u_{1,k})_{k=1}^{\infty}$ is bounded in $C^{2,1}(Q_2)$, there is a subsequence $(u_{2,k})\subset (u_{1,n})$ and $u_2^*\in C(Q_2)$ such that $u_{2,k}\to u_2^*$ in $C(Q_2)$ as $k\to\infty$.
Repeating this argument, we have a double sequence $(u_{j,k})$ and a sequence $(u_j^*)$ such that, for each $j\ge 1$, $u_{j,k}\to u_j^*$ in $C(Q_j)$ as $k\to\infty$.
We still denote $u_{n,n}$ by $u_n$, i.e., $u_n:=u_{n,n}$.
It is clear that $u_{j_1}^*\equiv u_{j_2}^*$ in $Q_{j_1}$ if $j_1\le j_2$.
Since $\RN\times (0,T)=\bigcup_{j=1}^{\infty}Q_j$, there is $u^*\in C(\RN\times (0,T))$ such that $u_n\to u^*$ in $C(K)$ as $n\to\infty$ for every compact set $K\subset\RN\times (0,T)$.
In particular,
\begin{equation}\label{L2E1}
u_n\to u^*\ \ \textrm{a.e. in}\ \ \RN\times(0,T).
\end{equation}
Let $w$ be a solution of (\ref{S1E2}).
It follows from Lemma~\ref{S3L2} that $|u_n(x,t)|\le w(x,t)$.
Since
\[
|G_t(x-y)u_n(y,t)|\le |G_t(x-y)w(y,t)|\ \ \textrm{for}\ \ y\in\RN,
\]
and
\[
G_t(x-y)w(y,t)\in L^1_y(\RN),
\]
by the dominated convergence theorem we see that
\begin{equation}\label{L2E2}
\lim_{n\to\infty}S(t)u_n
=\lim_{n\to\infty}\int_{\RN}G_t(s-y)u_n(y,t)dy
=\int_{\RN}G_t(s-y)u^*(y,t)dy=S(t)u^*.
\end{equation}
By (\ref{S3L1E2}) and (\ref{S3L1E6}) we see that if $T>0$ is small, then
\[
\int_0^t\int_{\RN}G_{t-s}(x-y)f(w(y,s))dyds\le Cg^{-1}(S(t)g(\phi))<\infty
\]
for each $(x,t)\in\RN\times (0,T)$,
and hence $G_{t-s}(x-y)f(w(y,s))\in L^1_{(y,s)}(\RN\times (0,T))$.
Since
\[
|G_{t-s}(x-y)f(u_{n-1}(y,s))|\le |G_{t-s}(x-y)f(w(y,s))|\ \ \textrm{for a.e.}\ (y,s)\in\RN\times (0,T)
\]
and
\[
G_{t-s}(x-y)f(w(y,s))\in L^1_{(y,s)}(\RN\times (0,T)),
\]
by the dominated convergence theorem we see that
\begin{multline}\label{L2E3}
\lim_{n\to\infty}\int_0^tS(t-s)f(u_{n-1}(s))ds
=\lim_{n\to\infty}\int_0^t\int_{\RN}G_{t-s}(x-y)f(u_{n-1}(y,s))dyds\\
=\int_0^t\int_{\RN}G_{t-s}(x-y)f(u^*(y,s))dyds
=\int_0^tS(t-s)f(u^*(s))ds.
\end{multline}
Thus, we take a limit of $u_n=\calF[u_{n-1}]$.
By (\ref{L2E1}), (\ref{L2E2}) and (\ref{L2E3}) we see that $u^*(t)=\calF[u^*](t)$ for $0<t<T$.

Since $|u_n|\le w$, we see that $|u^*|\le w$.
Since $|u^*|\le w$ in $\RN\times(0,T)$, by (\ref{S3L1E7-}) and the arbitrariness of $C_0>0$ we have
\[
\left\| u^*(t)-S(t)\phi\right\|_1
=\left\|\int_0^tS(t-s)f(u^*(s))ds\right\|_1
\le\left\|\int_0^tS(t-s)f(w(s))ds\right\|_1
\to 0\ \ \textrm{as}\ \ t\downarrow 0.
\]
Then, $\left\|u^*(t)-\phi\right\|_1\le\left\|u^*(t)-S(t)\phi\right\|_1+\left\|S(t)\phi-\phi\right\|_1\to 0$ as $t\downarrow 0$.
Since $\left\|\int_0^tS(t-s)f(w(s))\right\|_1<\infty$ for $0<t<T$, we can show by a similar way to the proof of Lemma~\ref{S3L1} that $u^*(t)\in C((0,T),L^1(\RN))$.
Thus, $u^*(t)\in C([0,T),L^1(\RN))$, and hence $u^*(t)$ is a mild solution.
Since $|u^*(t)|\le w(t)$, by Lemma~\ref{S3L1} we have (\ref{AE0}).
The proof of (i) is complete.
\end{proof}

\section{Nonexistence}
Let $0\le q<N/2$ be fixed.
Then there is $0<\e<N/2-q$.
We define $\phi_0$ by
\begin{equation}\label{S4E1}
\phi_0(x):=
\begin{cases}
|x|^{-N}\left(-\log|x|\right)^{-\frac{N}{2}-1+\e} & \textrm{if}\ |x|<1/e,\\
0 & \textrm{if}\ |x|\ge 1/e.
\end{cases}
\end{equation}
\begin{lemma}\label{S4L1}
Let $0\le q<N/2$, and let $\phi_0$ be defined by (\ref{S4E1}).
Then the following hold:\\
(i) $\phi_0\in X_q(\subset L^1(\RN))$.\\
(ii) The function $\phi_0$ does not satisfy (\ref{S2P2E0}) for any $T>0$.
\end{lemma}
\begin{proof}
(i) We write $\phi_0(r)=r^{-N}\left(-\log r\right)^{-N/2-1+\e}$ for $0<r<1/e$.
Since $\log (e+s)\le 1+\log s$ for $s\ge 0$, we have
\begin{equation}\label{S4L1E1}
\log (e+|\phi_0|)\le 1-N\log r-\left(\frac{N}{2}+1-\e\right)\log(-\log r)\le
-2N\log r
\end{equation}
for $0<r<1/e$.
Let $B(\tau):=\{x\in\RN\ |\ |x|<\tau\}$.
Using (\ref{S4L1E1}), we have
\begin{multline}\label{S4L1E3}
\int_{B(1/e)}|\phi_0|\left[\log(e+|\phi_0|)\right]^qdx
\le\omega_{N-1}\int_0^{1/e}\frac{(2N)^q(-\log r)^qr^{N-1}dr}{r^N(-\log r)^{N/2+1-\e}}\\
\le (2N)^q\omega_{N-1}\int_0^{1/e}\frac{dr}{r\left(-\log r\right)^{N/2+1-q-\e}}
=\frac{(2N)^q\omega_{N-1}}{\frac{N}{2}-q-\e}<\infty,
\end{multline}
where $\omega_{N-1}$ denotes the area of the unit sphere $\mathbb{S}^{N-1}$ in $\RN$.
By (\ref{S4L1E3}) we see that $\phi_0\in X_q$.\\
(ii) Suppose the contrary, i.e., there exists $\gamma_0>0$ such that (\ref{S2P2E0}) holds.
When $0<\tau<1/e$, we have
\[
\int_{B(\tau)}\phi_0(x)dx
=\omega_{N-1}\int_0^{\tau}\frac{dr}{r(-\log r)^{N/2+1-\e}}\\
=\frac{C}{(-\log\tau)^{N/2-\e}},
\]
where $C>0$ is independent of $\tau$.
Then,
\[
\gamma_0
\ge\frac{\int_{B(\tau)}\phi_0(x)dx}{(-\log\tau)^{-N/2}}\ge C(-\log\tau)^{\e}\to\infty
\ \ \textrm{as}\ \ \tau\downarrow 0.
\]
which is a contradiction.
Thus, the conclusion holds.
\end{proof}
\begin{proof}[Proof of Theorem~\ref{A}~(ii)]
Let $0\le q<N/2$.
It follows from Lemma~\ref{S4L1}~(i) that $\phi_0\in X_q$.
By Lemma~\ref{S4L1}~(ii) we see that there does not exist $\gamma_0>0$ such that (\ref{S2P2E0}) holds.
By Proposition~\ref{S2P2} the problem (\ref{S1E1}) with $\phi_0$ has no nonnegative integral solution.
\end{proof}

\section{Uniqueness}
\begin{proof}[Proof of Theorem~\ref{B}]
Let $q>N/2$.
Suppose that (\ref{S1E1}) has two integral solutions $u(t)$ and $v(t)$.
Using Young's inequality and the inequality $\left\|u(t)\right\|_{\infty}\le Ct^{-N/2}(-\log t)^{-q}$, we have
\begin{align*}
\left\|u(t)-v(t)\right\|_1
&\le\int_0^t\left\|G_{t-s}*\left\{\left(p|u|^{p-1}+p|v|^{p-1}\right)(u-v)\right\}\right\|_1ds\\
&\le p\int_0^t\left\|G_{t-s}\right\|_1\left(\left\|u\right\|_{\infty}^{p-1}+\left\|v\right\|^{p-1}_{\infty}\right)ds\sup_{0\le s\le t}\left\|u(s)-v(s)\right\|_1\\
&\le C\int_0^t\frac{ds}{\left\{s^{N/2}(-\log s)^q\right\}^{p-1}}\sup_{0\le s\le t}\left\|u(s)-v(s)\right\|_1.
\end{align*}
Since
\[
\int_0^t{s^{-N(p-1)/2}(-\log s)^{-(p-1)q}}ds=\frac{N(-\log t)^{1-2q/N}}{2q-N}
\]
and $1-2q/N<0$, we can choose $T>0$ such that $C\int_0^t{s^{-N(p-1)/2}(-\log s)^{-(p-1)q}}ds<1/2$ for every $0\le t\le T$.
Then, we have
\[
\sup_{0\le t\le T}\left\|u(t)-v(t)\right\|_1\le \frac{1}{2}\sup_{0\le s\le T}\left\|u(s)-v(s)\right\|_1,
\]
which implies the uniqueness.
\end{proof}

\noindent
{\bf Acknowledgements}\\
The author was supported by JSPS KAKENHI Grant Number 19H01797.

\end{document}